\documentclass[a4paper,reqno]{amsart}
\usepackage{amsfonts}
\usepackage{amsaddr}
\usepackage{color}
\usepackage{hyperref}
\definecolor{ddorange}{rgb}{1,0.5,0}
\definecolor{ddcyan}{rgb}{0,0.2,1.0}
\usepackage[textwidth=408pt,textheight=658pt,heightrounded, headsep=12pt,vmarginratio=1:1]{geometry}
\usepackage[T1]{fontenc} 
\usepackage[utf8]{inputenc}
\usepackage{amsmath,amssymb,amsthm,mathrsfs,esint}
\usepackage{mathtools}

\newcommand{\R}{{\mathbb R}}
\newcommand{\C}{{\mathbb C}}
\newcommand{\N}{{\mathbb N}}
\newcommand{\Z}{{\mathbb Z}}
\newcommand{\km}{{k^{-1}}}

\newcommand{\Rn}{{\R}^n}
\newcommand{\xy}{^\xi_y}
\newcommand{\Dxy}{{(\Omega\setminus A)^\xi_y}}
\newcommand{\Oxy}{{\Omega^\xi_y}}
\newcommand{\Axy}{{A^\xi_y}}
\newcommand{\dx}{\, \mathrm{d} x}
\renewcommand{\dh}{\, \mathrm{d} \mathcal{H}^{n-1}}
\newcommand{\hn}{\mathcal{H}^{n-1}}
\newcommand{\ho}{\mathcal{H}^0}
\newcommand{\Sn}{{\mathbb{S}^{n-1}}}
\newcommand{\ol}{\overline}

\newcommand{\sm}{\setminus}
\newcommand{\Mnn}{{\mathbb{M}^{n\times n}_{sym}}}
\newcommand{\dod}{{\partial_D \Omega}}
\newcommand{\don}{{\partial_N \Omega}}
\newcommand{\dom}{{\partial \Omega}}
\newcommand{\weak}{\rightharpoonup}
\newcommand{\wstar}{\stackrel{*}\rightharpoonup}

\DeclareMathOperator*{\aplim}{ap\,lim}

\theoremstyle{plain}
\begingroup
\theoremstyle{plain}
\newtheorem{theorem}{Theorem}[section]

\newtheorem{proposition}[theorem]{Proposition}
\newtheorem{lemma}[theorem]{Lemma}
\theoremstyle{definition}
\newtheorem{definition}[theorem]{Definition}
\theoremstyle{remark}
\newtheorem{remark}[theorem]{Remark}

\endgroup

\numberwithin{equation}{section}
\title[Compactness and lower semicontinuity 
in $GSBD$]{
Compactness and lower semicontinuity 
in $GSBD$}
\author{Antonin Chambolle \and Vito Crismale}
\address{CMAP, \'Ecole Polytechnique, CNRS, 91128 Palaiseau Cedex, France}
\email[Antonin Chambolle]{antonin.chambolle@cmap.polytechnique.fr}
\email[Vito Crismale]{vito.crismale@polytechnique.edu}

\begin{document}
\begin{abstract}
In this paper we prove a compactness and semicontinuity
result in $GSBD$ for sequences with bounded Griffith energy.
This generalises classical results in $(G)SBV$ by Ambrosio~\cite{Amb89UMI, Amb90GSBV, AmbNew95} and $SBD$~by Bellettini-Coscia-Dal Maso~\cite{BelCosDM98}.
As a result, the static problem in Francfort-Marigo's variational
approach to crack growth~\cite{FraMar98} admits (weak) solutions.
\smallskip

\noindent
\end{abstract}

\maketitle
\tableofcontents

\section{Introduction}

The variational approach to fracture was introduced by Francfort and
Marigo in~\cite{FraMar98} in order to build crack evolutions in brittle
materials, following Griffith's laws~\cite{Griffith}, without 
\textit{a priori} knowledge of the crack path (or surface in higher
dimension). It relies on successive minimisations of the
\emph{Griffith energy}:
\[
(u,  K)  \mapsto \int\limits_{\Omega\setminus K} \C e(u):e(u)dx + \gamma\,\hn(K)
\]
where $\Omega\subset\R^n$ is a bounded open set, the \textit{reference
configuration}, $u:\Omega\to \R^n$ is an (infinitesimal) \emph{displacement},
$e(u)$ its symmetrised gradient
(the \textit{infinitesimal elastic strain})
 and $\C$ the \emph{Cauchy stress tensor} defining the 
\textit{Hooke's law} (in particular, $\C a:a$ defines a positive
definite quadratic form of the $n\times n$ symmetric tensor $a$).
The symmetrised gradient  $e(u)$ is defined
out of the \emph{crack set} $K$, which is in the theory a compact $(n{-}1)$-dimensional
set and is penalised by its surface (multiplied by a coefficient
$\gamma$ called the \textit{toughness}).

The minimisation of the energy is under the constraint that
$K$ should contain a previously computed crack $K_0$, and that $u$ should
satisfy a Dirichlet condition $u=u_0$ on a subset $\dod\setminus K$
of $\partial \Omega$, where $\dod$ is a regular part of the boundary
and $u_0$ a sufficiently regular displacement. Hence an important
question in the theory is whether the problem
\begin{equation}\label{eq:FM}
\min_{\substack{u=u_0\,\textup{on}\, \dod \setminus K\\ K_0 \subset K \text{compact}}} \hspace{1em}
\int\limits_{\Omega\setminus K} \C e(u):e(u)dx + \gamma\,\hn(K)
\end{equation}
has a solution.

This problem however is not easy to analyse,
since the energy controls very little of the function $u$: 
for instance
if $K$ almost cuts out from $\dod$ a connected component of $\Omega$, the function $u$
may have any (arbitrarily large) value in this component at 
small cost. 
 
From a technical point of view, one cannot take truncations or compositions with bounded transformations to get an \emph{a priori} $L^\infty$ bound for minimisers.
%
In fact, the integrability of $e(u)$ is in general lost by $e(\psi(u))$, unless
$\psi(y)=y_0 + \lambda y$, for some $y_0 \in \Rn$, $\lambda\in \R$ (see e.g.\ the introduction of \cite{DM13}). 


For this reason, most of the ``sound'' approaches to 
problem
\eqref{eq:FM} consider additional assumptions. 
In particular, a global $L^\infty$ bound on the
displacements ensures one may work in the class $SBD$ of 
\emph{Special
functions with Bounded Deformation}~\cite{AmbCosDM97}, provided one considers a \emph{weak formulation}
of the problem where $K$ is replaced with the intrinsic jump set $J_u$
 of $u$ (which needs not to be closed anymore): in this space minimising
sequences are shown to be compact~\cite{BelCosDM98}, and the energy
to be lower semicontinuous. Another possible assumption
is, in 2$d$, that the crack set $K$  be  connected~\cite{DMToa02,Cha03}.

The natural space for studying~\eqref{eq:FM}, in fact, is not $SBD(\Omega)$
(which assumes that the symmetrised gradient of $u$ is a measure and
hence $u$ is in  $L^{n/(n-1)}(\Omega;\R^n)$)
but the space $GSBD(\Omega)$,
introduced by Dal~Maso in~\cite{DM13}. This space, defined by
the slicing properties of the functions, is designed in order to
contain ``all'' displacements $u$ for which the energy is finite. 
 Even if \cite{DM13} proves compactness under very mild assumptions on the integrability of displacements,  no compactness result was available in $GSBD$ for minimizing sequences of  (the weak formulation of) \eqref{eq:FM}  until very recently.

The first existence result 
without further
constraint has been proven indeed in~\cite{FriSol16}, in \emph{dimension two}. 
It relies on a delicate
construction showing a 
\emph{piecewise Korn
inequality}, in \cite{FriPWKorn} (for approximated Korn and Korn-Poincaré inequalities see also e.g.\ \cite{ CFI16ARMA, CCF16, Fri17M3AS}, for piecewise rigidity cf.\ \cite{ChaGiaPon07}).

In this paper, we prove the following general compactness result  for sequences bounded
in energy, in the space $GSBD(\Omega)$, in any dimension.  
\begin{theorem}\label{thm:main}
Let $\phi \colon \R^+ \to \R^+$ be a non-decreasing function with
\begin{equation}\label{eq:equiint}
\lim_{t\to +\infty} \frac{\phi(t)}{t}=+\infty\,,
\end{equation}
and let $(u_h)_h$ be a sequence in $GSBD(\Omega)$ such that
\begin{equation}\label{eq:boundGSBD}
\int\limits_\Omega \phi\big( |e(u_h)| \big) \dx + \hn(J_{u_h}) < M\,,
\end{equation}
for some constant $M$ independent of $h$. 
Then there exists a subsequence, still denoted by $(u_h)_h$, such that 
\begin{equation}\label{eq:defA}
A:=\{x\in \Omega \colon\, |u_h(x)|\rightarrow +\infty\}
\end{equation} 
has finite perimeter, and 
%
%
$u\in GSBD(\Omega)$ with $u=0$ on $A$ 
for which
\begin{subequations}\label{1407180906}
\begin{align}
u_h \rightarrow u  \quad\quad\,&\mathcal{L}^n\text{-a.e.\ in } \Omega \sm A\,,\label{eq:convmisura}\\
e(u_h)  \rightharpoonup e(u) \quad &\text{in } L^1(\Omega\setminus A; \Mnn)\,,\label{eq:convGradSym}\\
\hn(J_u \cup \partial^* A )\leq \liminf_{h\to \infty} \,&\hn(J_{u_h})\,.\label{eq:sciSalto}
\end{align}
\end{subequations}
\end{theorem}
The proof of this theorem
is in our opinion simpler than~\cite{FriSol16}, 
 even if a fundamental tool is 
a quite technical Korn-Poincar\'e inequality for functions
with small jump set, proved in~\cite{CCF16} and employed also in \cite{ChaConFra18ARMA, ChaConIur17, ChaCri17}.
We combine this inequality with arguments in the spirit of Rellich's type compactness theorems. 

Theorem~\ref{thm:main} gives then the existence of minimisers for the Griffith energy with Dirichlet boundary conditions in the weak formulation (see Theorem~\ref{teo:existminGrif}), which by results in~\cite{CFI17DCL,ChaConIur17} satisfy the properties of strong solutions in the interior of $\Omega$. 
 In the forthcoming paper \cite{CC19} we prove existence of solutions for the strong formulation \eqref{eq:FM} by extending the regularity theorems in~\cite{CFI17DCL,ChaConIur17} up to the boundary, when $\dod$ is of class $C^1$ and $u_0$ is Lipschitz.


The major issue for establishing the compactness result of Theorem~\ref{thm:main} comes from the lack of control on both the displacement and its full gradient, as is natural in the study of brittle fracture in small strain  (linearised) elasticity~\cite{Griffith}. 

A bound such as \eqref{eq:boundGSBD} for the full gradient in place of the symmetrised gradient is available 
 for brittle fractures models in finite strain elasticity or in small strain elasticity in the simplified \emph{antiplane case} (i.e.\ when the displacement $u$ is vertical and depends only on the horizontal components). 
In these cases, the energy is closely related to the \emph{Mumford-Shah functional} in image reconstruction \cite{MumSha} (which however includes a fidelity term, artificial from a mechanical standpoint). 
In this  context, the original strategy of passing through a weak formulation in terms of $u$ was first proposed 
 by De Giorgi and realised by Ambrosio \cite{Amb89UMI, Amb90GSBV, Amb94, AmbNew95}, for the existence of weak solutions, and De Giorgi, Carriero, Leaci in \cite{DeGCarLea} (see also e.g.\ \cite{CarLea91, FonFus97}), for the regularity giving the improvement to strong solutions (an alternative approach, where the discontinuity set is the main variable, has been successfully employed in \cite{DMMorSol92, MadSol01}).


Ambrosio's results are obtained in the space $GSBV$ \cite{DeGioAmb88GBV}, and have been extended to $GSBD$ by Dal~Maso in~\cite{DM13}. In both cases, a control of the
values is required to obtain compactness, guaranteeing that the set $A$ in Theorem~\ref{thm:main} is empty.
{
Without such a control, it is still relatively 
simple to obtain a $GSBV$ version of Theorem~\ref{thm:main}.
For instance, in the scalar case
one can consider as in~\cite{Amb89UMI} the sequences of truncated functions
$u_k^N:=\max\{ -N,\min\{ u_k, N\}\}$ for any integer $N\ge 1$, which
are compact in $BV$ and converge up to subsequences. Then, by a diagonal
argument, sending then $N$ to $+\infty$, one builds a subsequence 
$(u_{k_h})_h$ which converges a.e.\ to some $u$, except on a possible set 
$A$ where it goes to $+\infty$ or $-\infty$.
The scalar version of~\eqref{eq:convGradSym} 
is obtained exactly as in~\cite{Amb89UMI} (see in particular~\cite[Prop.~4.4]{Amb89UMI}), 
considering perturbations $w\in L^1(\Omega)$ with $w=0$ a.e.~in $A$.
One possible way to derive inequality~\eqref{eq:sciSalto} 
is then by slicing arguments, similar to (but simpler than) the 
arguments  in Section~\ref{Sec:proofMain} of  the current paper. 
The extension to the vectorial case is not difficult in $GSBV$.
}
%
%
%
%
%
%
%
%

This strategy however fails in our case since, as already mentioned, the
space $GSBD$ is not stable by truncations. The way out  to get compactness without any assumption on the displacements  is to locally approximate $GSBD$ functions with piecewise infinitesimal rigid motions, by means of the Korn-Poincaré inequality in \cite{CCF16},
and use that such motions belong to a finite dimensional space.  We then obtain compactness with respect to the convergence in $\mathcal{L}^n$-measure, but still, we can not exclude the existence of a set $A$ of points where the limit is not in $\Rn$. A slicing argument
then is used to show that $A$ has finite perimeter, whose measure is controlled by \eqref{eq:sciSalto}.
(Existence for~\eqref{eq:FM} is then deduced by considering the
limit of a minimising sequence and setting in $A$
the limit function equal to 0, or to any ground state of the elastic energy.)



A more general (and difficult) 
approach, for $GSBV^p$, has been 
proposed by Friedrich in \cite{Fri19}: 
there, the set $A$ is a priori 
removed by a careful modification at the level of the minimising sequence, with a control of
the energy. 
Friedrich and Solombrino also prove in~\cite{FriSol16} existence of quasistatic evolutions in dimension two, extending in that case the antiplane result by Francfort and Larsen in \cite{FraLar03}, 
 (see \cite{BabGia14} for the existence of \emph{strong} quasistatic evolutions in dimension two, and e.g.\ \cite{DMFraToa07, DMLaz10} for quasistatic evolutions for brittle fractures with finite strain elasticity).  

\section{Notation and preliminaries}\label{Sec1}
For every $x\in \Rn$ and $\varrho>0$ let $B_\varrho(x)$ be the open ball with center $x$ and radius $\varrho$. For $x$, $y\in \Rn$, we use the notation $x\cdot y$ for the scalar product and $|x|$ for the norm.
We denote by $\mathcal{L}^n$ and $\mathcal{H}^k$ the $n$-dimensional Lebesgue measure and the $k$-dimensional Hausdorff measure. For any locally compact subset $B$ of $\Rn$, the space of bounded $\R^m$-valued Radon measures on $B$ is denoted by $\mathcal{M}_b(B;\R^m)$. For $m=1$ we write $\mathcal{M}_b(B)$ for $\mathcal{M}_b(B;\R)$ and $\mathcal{M}^+_b(B)$ for the subspace of positive measures of $\mathcal{M}_b(B)$. For every $\mu \in \mathcal{M}_b(B;\R^m)$, its total variation is denoted by $|\mu|(B)$.
We write $\chi_E$ for the indicator function of any $E\subset \R^n$, which is 1 on $E$ and 0 otherwise. 
 We call \emph{infinitesimal rigid motion} any affine function
with skew-symmetric gradient. Let us also set $\widetilde{\R}:=\R \cup \{-\infty, +\infty\}$ and $\R^*:= \R \sm \{0\}$. 
\begin{definition}\label{def:aplim}
Let $E\subset \Rn$, $v\colon E \to \R^m$ an $\mathcal{L}^n$-measurable function, $x\in \Rn$ such that
\begin{equation*}
\limsup_{\varrho\to 0^+}\frac{\mathcal{L}^n(E\cap B_\varrho(x))}{\varrho^n}>0\,.
\end{equation*}
A vector $a\in \Rn$ is the \emph{approximate limit} of $v$ as $y$ tends to $x$ if for every $\varepsilon>0$
\begin{equation*}
\lim_{\varrho \to 0^+}\frac{\mathcal{L}^n(E \cap B_\varrho(x)\cap \{|v-a|>\varepsilon\})}{\varrho^n}=0\,,
\end{equation*}
and then we write
\begin{equation}\label{3105171542}
\aplim \limits_{y\to x} v(y)=a\,.
\end{equation}
\end{definition}
\begin{remark}\label{rem:3105171601}
Let $E$, $v$, $x$, and $a$ be as in Definition~\ref{def:aplim} and let $\psi$ be a homeomorphism between $\R^m$ and a bounded open subset of $\R^m$. Then \eqref{3105171542} holds if and only if \begin{equation*}
\lim_{\varrho\to 0^+}\frac{1}{\varrho^n}\hspace{-1em}\int \limits_{E\cap B_\varrho(x)}\hspace{-1em}|\psi(v(y))-\psi(a)|\,\mathrm{d}y=0\,.
\end{equation*}
\end{remark}

\begin{definition}
Let $U\subset \Rn$ open, and $v\colon U\to \R^m$ be $\mathcal{L}^n$-measurable. The \emph{approximate jump set} $J_v$ is the set of points $x\in U$ for which there exist $a$, $b\in \R^m$, with $a \neq b$, and $\nu\in \Sn$ such that
\begin{equation*}
\aplim\limits_{(y-x)\cdot \nu>0,\, y \to x} v(y)=a\quad\text{and}\quad \aplim\limits_{(y-x)\cdot \nu<0, \, y \to x} v(y)=b\,.
\end{equation*}
The triplet $(a,b,\nu)$ is uniquely determined up to a permutation of $(a,b)$ and a change of sign of $\nu$, and is denoted by $(v^+(x), v^-(x), \nu_v(x))$. The jump of $v$ is the function 
defined by $[v](x):=v^+(x)-v^-(x)$ for every $x\in J_v$. Moreover, we define
\begin{equation}
J_v^1:=\{x\in J_v \colon |[v](x)|\geq 1\}\,.
\end{equation}
\end{definition}
\begin{remark}
By Remark~\ref{rem:3105171601}, $J_v$ and $J^1_v$ are Borel sets and $[v]$ is a Borel function. By Lebesgue's differentiation theorem, it follows that $\mathcal{L}^n(J_v)=0$.
\end{remark}

\par
\medskip
\paragraph{\bf $BV$ and $BD$ functions.}
If $U\subset \Rn$ open, a function $v\in L^1(U)$ is a \emph{function of bounded variation} on $U$, and we write $v\in BV(U)$, if $\mathrm{D}_i v\in \mathcal{M}_b(U)$ for $i=1,\dots,n$, where $\mathrm{D}v=(\mathrm{D}_1 v,\dots, \mathrm{D}_n v)$ is its distributional gradient. A vector-valued function $v\colon U\to \R^m$ is in $BV(U;\R^m)$ if $v_j\in BV(U)$ for every $j=1,\dots, m$.
The space $BV_{\mathrm{loc}}(U)$ is the space of $v\in L^1_{\mathrm{loc}}(U)$ such that $\mathrm{D}_i v\in \mathcal{M}_b(U)$ for $i=1,\dots,n$. 

A $\mathcal{L}^n$-measurable bounded set $E\subset \R^n$ is a set of \emph{finite perimeter} if $\chi_E$ is a function of bounded variation. The \emph{reduced boundary} of $E$, denoted by $\partial^*E$, is the set of points $x\in \mathrm{supp}\, |\mathrm{D}\chi_E|$ such that the limit $\nu_E(x):=\lim_{\varrho \to 0^+}\frac{\mathrm{D}\chi_E(B_\varrho(x))}{|\mathrm{D}\chi_E|(B_\varrho(x))}$ exists and satisfies $|\nu_E(x)|=1$. The reduced boundary is countably $(\hn, n-1)$ rectifiable, and the function $\nu_E$ is called \emph{generalised inner normal} to $E$.

A function $v\in L^1(U;\Rn)$ belongs to the space of \emph{functions of bounded deformation} if its distributional symmetric gradient $\mathrm{E}v$ belongs to $\mathcal{M}_b(U;\Rn)$.
It is well known (see \cite{AmbCosDM97, Tem}) that for $v\in BD(U)$, $J_v$ is countably $(\hn, n-1)$ rectifiable, and that
\begin{equation}
\mathrm{E}v=\mathrm{E}^a v+ \mathrm{E}^c v + \mathrm{E}^j v\,,
\end{equation}
where $\mathrm{E}^a v$ is absolutely continuous with respect to $\mathcal{L}^n$, $\mathrm{E}^c v$ is singular with respect to $\mathcal{L}^n$ and such that $|\mathrm{E}^c v|(B)=0$ if $\hn(B)<\infty$, while $\mathrm{E}^j v$ is concentrated on $J_v$. The density of $\mathrm{E}^a v$ with respect to $\mathcal{L}^n$ is denoted by $e(v)$, and we have that (see \cite[Theorem~4.3]{AmbCosDM97} and recall \eqref{def:aplim}) for $\mathcal{L}^n$-a.e.\ $x\in U$
\begin{equation}\label{3105171931}
\aplim\limits_{y\to x} \frac{\big(v(y)-v(x)-e(v)(x)(y-x)\big)\cdot (y-x)}{|y-x|^2}=0\,.
\end{equation}
The space $SBD(U)$ is the subspace of all functions $v\in BD(U)$ such that $\mathrm{E}^c v=0$, while for $p\in (1,\infty)$
\begin{equation*}
SBD^p(U):=\{v\in SBD(U)\colon e(v)\in L^p(\Omega;\Mnn),\, \hn(J_v)<\infty\}\,.
\end{equation*}
Analogous properties hold for $BV$, as the countable rectifiability of the jump set and the decomposition of $\mathrm{D}v$, and the spaces $SBV(U;\R^m)$ and $SBV^p(U;\R^m)$ are defined similarly, with $\nabla v$, the density of $\mathrm{D}^a v$, in place of $e(v)$.
For a complete treatment of $BV$, $SBV$ functions and $BD$, $SBD$ functions, we refer to \cite{AFP} and to \cite{AmbCosDM97, BelCosDM98, Bab15, Tem}, respectively.
\par
\medskip
\paragraph{\bf $GBD$ functions.}
We now recall the definition and the main properties of the space $GBD$ of \emph{generalised functions of bounded deformation}, introduced in \cite{DM13}, referring to that paper for a general treatment and more details. Since the definition of $GBD$ is given by slicing (differently from the definition of $GBV$, \textit{cf.}~\cite{DeGioAmb88GBV, Amb90GSBV}), we introduce before some notation.

Fixed $\xi \in \Sn:=\{\xi \in \Rn\colon |\xi|=1\}$, for any $y\in \Rn$ and $B\subset \Rn$ let
\begin{equation*}
\Pi^\xi:=\{y\in \Rn\colon y\cdot \xi=0\},\qquad B^\xi_y:=\{t\in \R\colon y+t\xi \in B\}\,,
\end{equation*}
and for every function $v\colon B\to \R^n$ and $t\in B^\xi_y$ let
\begin{equation*}
v^\xi_y(t):=v(y+t\xi),\qquad \widehat{v}^\xi_y(t):=v^\xi_y(t)\cdot \xi\,.
\end{equation*}

\begin{definition}[\cite{DM13}]\label{def:GBD}
Let $\Omega\subset \Rn$ be bounded and open, and $v\colon \Omega\to \Rn$ be $\mathcal{L}^n$-measurable. Then $v\in GBD(\Omega)$ if there exists $\lambda_v\in \mathcal{M}^+_b(\Omega)$ such that  one of the following equivalent conditions holds true  for every $\xi \in \Sn$:
\begin{itemize}
\item[(a)] for every $\tau \in C^1(\R)$ with $-\tfrac{1}{2}\leq \tau \leq \tfrac{1}{2}$ and $0\leq \tau'\leq 1$, the partial derivative $\mathrm{D}_\xi\big(\tau(v\cdot \xi)\big)=\mathrm{D}\big(\tau(v\cdot \xi)\big)\cdot \xi$ belongs to $\mathcal{M}_b(\Omega)$, and for every Borel set $B\subset \Omega$ 
\begin{equation*}
\big|\mathrm{D}_\xi\big(\tau(v\cdot \xi)\big)\big|(B)\leq \lambda_v(B);
\end{equation*}
\item[(b)] $\widehat{v}^\xi_y \in BV_{\mathrm{loc}}(\Omega^\xi_y)$ for $\hn$-a.e.\ $y\in \Pi^\xi$, and for every Borel set $B\subset \Omega$ 
\begin{equation}\label{3105171445}
\int \limits_{\Pi^\xi} \Big(\big|\mathrm{D} {\widehat{v}}_y^\xi\big|\big(B^\xi_y\setminus J^1_{{\widehat{v}}^\xi_y}\big)+ \mathcal{H}^0\big(B^\xi_y\cap J^1_{{\widehat{v}}^\xi_y}\big)\Big)\dh(y)\leq \lambda_v(B)\,,
\end{equation}
where
$J^1_{{\widehat{u}}^\xi_y}:=\left\{t\in J_{{\widehat{u}}^\xi_y} : |[{\widehat{u}}_y^\xi]|(t) \geq 1\right\}$.
\end{itemize} 
The function $v$ belongs to $GSBD(\Omega)$ if $v\in GBD(\Omega)$ and $\widehat{v}^\xi_y \in SBV_{\mathrm{loc}}(\Omega^\xi_y)$ for every $\xi \in \Sn$ and for $\hn$-a.e.\ $y\in \Pi^\xi$.
\end{definition}
$GBD(\Omega)$ and $GSBD(\Omega)$ are vector spaces, as stated in \cite[Remark~4.6]{DM13}, and one has the inclusions $BD(\Omega)\subset GBD(\Omega)$, $SBD(\Omega)\subset GSBD(\Omega)$, which are in general strict (see \cite[Remark~4.5 and Example~12.3]{DM13}).
For every $v\in GBD(\Omega)$
the \emph{approximate jump set} $J_v$ is still countably $(\hn,n-1)$-rectifiable (\textit{cf.}~\cite[Theorem~6.2]{DM13}) and can be reconstructed from the jump of the slices $\widehat{v}^\xi_y$ (\cite[Theorem~8.1]{DM13}).
Indeed, for every $C^1$ manifold $M\subset \Omega$ with unit normal $\nu$, it holds that for $\hn$-a.e.\ $x\in M$ there exist the \emph{traces} $v_M^+(x)$, $v_M^-(x)\in \Rn$ such that
\begin{equation}\label{0106172148}
\aplim \limits_{\pm(y-x)\cdot \nu(x)>0, \, y\to x} \hspace{-1em} v(y)=v_M^{\pm}(x)
\end{equation}
and they can be reconstructed from the traces of the one-dimensional slices (see \cite[Theorem~5.2]{DM13}).
Every $v\in GBD(\Omega)$ has an \emph{approximate symmetric gradient} $e(v)\in L^1(\Omega;\Mnn)$, characterised by \eqref{3105171931} and such that for every $\xi \in \Sn$ and $\hn$-a.e.\ $y\in\Pi^\xi$
\begin{equation}\label{3105171927}
e(v)^\xi_y \xi\cdot \xi=\nabla \widehat{v}^\xi_y \quad\mathcal{L}^1\text{-a.e.\ on }\Omega^\xi_y\,.
\end{equation}

By these properties of slices it follows that, if 
$v \in GSBD(\Omega)$ with $e(v)\in L^1(\Omega;\Mnn)$ and $\hn(J_v)<+\infty$, then for every Borel set $B\subset \Omega$
\begin{equation}\label{0101182248}
\hn(J_v\cap B)=(2 \omega_{n-1})^{-1} \int\limits_{\Sn} \bigg(\int \limits_{\Pi^\xi} \ho(J_{v\xy}\cap B\xy)\,\mathrm{d}\hn(y)\bigg)\,\mathrm{d}\hn(\xi)
\end{equation}
and the two conditions in the definition of $GSBD$ for $v$ hold for $\lambda_v \in \mathcal{M}_b^+(\Omega)$ such that
\begin{equation}\label{0201181308}
\lambda_v(B)\leq \int_B |e(v)| \dx + \hn(J_v \cap B)\,,
\end{equation} 
for every Borel set $B\subset \Omega$ (\textit{cf.}~also \cite[Theorem~1]{Fri17ARMA} and \cite[Remark~2]{Iur14}).

We now recall the following result, proven 
in~\cite[Proposition~2]{CCF16}. Notice that the proposition is therein stated in $SBD$, but the proof, which is based on the Fundamental Theorem of Calculus along lines, still holds for $GSBD$, with small adaptations.
%
%
\begin{proposition}[\cite{CCF16}]\label{prop:3CCF16}
Let $Q_r =(-r,r)^n$, $v\in GSBD(Q)$, $p\in [1,\infty)$. Then there exist a Borel set $\omega\subset Q_r$ and an affine function $a\colon \Rn\to\Rn$ with $e(a)=0$ such that 
\begin{equation*}
\mathcal{L}^n(\omega)\leq cr \hn(J_v)\,
\end{equation*} 
and
\begin{equation}\label{prop3iCCF16}
\int\limits_{Q_r\setminus \omega} |v-a|^{p} \dx\leq cr^{p} \int\limits_{Q_r} |e(v)|^p\dx\,.
\end{equation}
The constant $c$ depends only on $p$ and $n$.
\end{proposition}
 
We conclude the section with a technical lemma.
 \begin{lemma}\label{le:referee2}
Let $E \subset \Rn$ Borel, $v_h\colon E \to \Rn$ for every $h$, and consider the $n$ sequences $(v_h \cdot e_i)_h$, obtained by taking every component of $v_h$ with respect to the canonical basis of $\Rn$ $\{e_1,\dots,e_n\}$. Assume that every $(v_h \cdot e_i)_h$ converges pointwise $\mathcal{L}^n$-a.e.\ to a $v_i \colon E \to \widetilde{\R}$, and that for $\mathcal{L}^n$-a.e.\ $x \in E$ there is $i\in\{1,\dots,n\}$ for which $v_i(x) \in \{-\infty, +\infty\}$.
Then for $\hn$-a.e.\ $\xi \in \Sn$

\begin{equation}\label{1307181345}
|v_h \cdot \xi| \rightarrow +\infty \quad \mathcal{L}^n\text{-a.e.\ in } E \,.
\end{equation}
\end{lemma}

\begin{proof}
On the sets
\[
E_i:=\{|v_h\cdot e_i| \rightarrow +\infty\} \cap \bigcap_{j\neq i}\{ 
\limsup_{h\to\infty} (|v_h\cdot e_j|/|v_h\cdot e_i|) < +\infty\}\,,
\]
we have that \eqref{1307181345} holds for every $\xi$ in $\{\xi\in \Sn \colon \, \xi_i\neq 0\}$, which is of full $\hn$ measure in $\Sn$. 

Let us thus consider the case when there are $m$ components of $v_h$, with $1<m\leq n$, that we may assume up to a permutation
$v_h \cdot e_1, \dots,\, v_h \cdot e_m$, such that $\frac{v_h \cdot e_i}{v_h \cdot e_j} \to \xi_{i,j} \in \R^*$ for $1\leq i < j \leq m$ and $|\frac{v_h \cdot e_i}{v_h \cdot e_j}| \to +\infty$ for $i \in \{1,\dots,m\}$ and $j \in\{m+1, \dots, n\}$ (if $m<n$).
In this case \eqref{1307181345} does not hold only for
\[
\Sn \cap (1,\xi_{1,2}^{-1}, \dots, \xi_{1,m}^{-1},0\dots,0)^\perp\,,
\]
which has dimension $n-2$. 
Notice now that for every $m$ for which $m$ components go faster to infinity than the other ones, there is an at most countable collection of $(\xi_{1,2},\dots,\xi_{1,m}) \in (\R^*)^{m-1}$ for which $\frac{v_h \cdot e_1}{v_h \cdot e_j} \to \xi_{1,j}$ for $j\in \{2,\dots,m\}$ on a subset of $E$ of positive $\mathcal{L}^n$ measure.
Thus \eqref{1307181345} holds for every $\xi$ except on an at most countable union of $\hn$-negligible sets of $\Sn$.
\end{proof}

\section{The main compactness and lower semicontinuity result}\label{Sec:proofMain}

In this section we prove Theorem~\ref{thm:main}, the main result of the paper.

%
%

\begin{proof}[Proof of Theorem~\ref{thm:main}]
 We divide the proof into three parts: compactness (with respect to the convergence in measure, by means of approximation through piecewise infinitesimal rigid motions), lower semicontinuity, and closure (in $GSBD$).

\par
\medskip
\paragraph{\bf Compactness.}

For every $k\in \N$ 
and  $z\in (2 \km) \Z^n $ we consider the cubes of center $z$
\begin{equation*}
\begin{split}
q_{k,z}:=z+(-\km,\km)^n.
\end{split}
\end{equation*}
Then $\Omega_k :=\Omega\setminus \bigcup_{q_{k,z}\not\subset\Omega}\overline{q_{k,z}}$ is essentially the union of the cubes which are contained in $\Omega$.

 We apply Proposition~\ref{prop:3CCF16}
with $p=1$ 
in any $q_{k,z}\subset\Omega$, so for $r=k^{-1}$. %
 Then there exist
 sets $\omega_{k,z}^h\subset q_{k,z}$ with 
\begin{equation}\label{1005171230}
\mathcal{L}^n(\omega_{k,z}^h)\leq c k^{-1} \hn(J_{u_h}\cap q_{k,z}) 
\end{equation}
 and affine functions $a_{k,z}^h\colon \Rn \to \Rn$, with $e(a_{k,z}^h)=0$, such that
\begin{equation}\label{prop3iCCF16applicata}
\int\limits_{q_{k,z}\setminus \omega_{k,z}^h}|u_h-a_{k,z}^h| \dx\leq c\,k^{-1}\int\limits_{q_{k,z}}|e(u_h)|\dx\,.
\end{equation}

The functions $(a_{k,z}^h)_{h\ge 1}$ belong to the finite dimensional space of affine functions.  For any sequence of the $i$-th component $(a_{k,z}^h\cdot e_i)_h$, $i=1,\dots,n$,
we have the following cases: 
\begin{itemize}
\item it is bounded, and then converges uniformly (up to a subsequence) to an affine function;
\item it is unbounded, and then one of the two alternative possibilites below occurs:
\begin{itemize}
\item it converges globally,
up to a subsequence, to $+\infty$ or $-\infty$;
\item there is a hyperplane $\{x\cdot\nu = t\}$ ($\nu\in \R^n$, $t\in\R$)
and a subsequence such that $a_{k,z}^h(x)\cdot e_i \to +\infty$ if
$x\cdot\nu>t$ and 
$a_{k,z}^h(x)\cdot e_i \to -\infty$ if $x\cdot\nu<t$.
\end{itemize}
 (To see this, consider the bounded sequence $\frac{a_{k,z}^h\cdot e_i}{\|a_{k,z}^h \cdot e_i\|}$, for any norm $\|\cdot\|$ on the space of affine functions, which has converging subsequences.)
\end{itemize} 


Let $\tau$ denote the function $\tanh$ (or any smooth, $1$-Lipschitz 
increasing function from $-1$ to $1$  with $\tau(0)=0$ ). As a consequence we obtain
that, up to a subsequence, the function 
\[
a^h_k(x):=\sum_{q_{z,k}\subset\Omega} a_{k,z}^h(x)\,\chi_{q_{k,z}}(x)
\]
is such that $\big(\tau(a^h_k \cdot e_i)\big)_h$ converges to some function 
in $L^1(\Omega_k)$, for any $i=1,\dots,n$.  Indeed, we have
\[
\tau(a^h_k \cdot e_i)(x)=\sum_{q_{z,k}\subset\Omega} \tau(a^h_{k,z}\cdot e_i)(x) \,\chi_{q_{k,z}}(x)\,,
\]
and in any cube $q_{k,z}$ the sequence $\big(\tau(a^h_{k,z}\cdot e_i)\big)_h$ converges uniformly either to a function valued in $(-1,1)$, if $(a^h_{k,z}\cdot e_i)_h$ is bounded, or to a function with values $-1$ and 1, attained where the limit of $(a^h_{k,z}\cdot e_i)_h$ is $+\infty$ or $-\infty$, respectively (notice that at this stage $k$ is fixed).

Clearly the subsequence could be extracted from a previous subsequence
built at the stage $k-1$, hence by a diagonal argument, we may assume
that for any $k$, $(\tau(a^h_k\cdot e_i))_h$ converges for all $i=1,\dots n$,
in $L^1(\Omega_k)$. 

We have that for each $i=1,\dots,n$, $k\ge 1$, and $l,m\ge 1$,
\begin{multline} \label{eq:split4}
\int\limits_\Omega |\tau(u_m\cdot e_i)-\tau(u_l\cdot e_i)|\dx
\le 2|\Omega\setminus \Omega_k|
+ \int\limits_{\Omega_k} |\tau(u_m\cdot e_i) - \tau(a^m_k\cdot e_i)| \dx
\\+ \int\limits_{\Omega_k} |\tau(a^m_k\cdot e_i) - \tau(a^l_k\cdot e_i)| \dx
+ \int\limits_{\Omega_k} |\tau(u_l\cdot e_i) - \tau(a^l_k\cdot e_i)| \dx.
\end{multline}
By construction,
\[
\lim_{l,m\to+\infty} 
 \int\limits_{\Omega_k} |\tau(a^m_k\cdot e_i) - \tau(a^l_k\cdot e_i)| \dx = 0.
\]
On the other hand,
\begin{align*}
 \int\limits_{\Omega_k} |\tau(u_m\cdot e_i) - \tau(a^m_k\cdot e_i)| \dx
& =\sum_{q_{k,z}\subset\Omega} \int\limits_{q_{k,z}} |\tau(u_m\cdot e_i) - \tau(a^m_{k,z}\cdot e_i)| \dx \\
& \le \sum_{q_{k,z}\subset\Omega} \bigg(
2|\omega_{k,z}^m| + \hspace{-1em}
\int\limits_{q_{k,z}\sm \omega_{k,z}^m} \hspace{-1em}|u_m- a^m_{k,z}| \dx\bigg) \\
& \le  \frac{2c}{k}\bigg(\hn (J_{u_m})  + \int\limits_{\Omega_k} |e(u_m)|\dx\bigg)
\le \frac{C}{k}.
\end{align*}
Using that $|\Omega\setminus \Omega_k|\to 0$ as $k\to\infty$,
we deduce from~\eqref{eq:split4} that $(\tau(u_h\cdot e_i))_h$ is
a Cauchy sequence (for each $i$) and therefore converges in $L^1(\Omega)$
to some limit which we denote $\tilde \tau_i$. Up to a further subsequence,
we may assume that the convergence occurs almost everywhere and, 
by \eqref{eq:equiint} and \eqref{eq:boundGSBD}, that $\big(e(u_h)\big)_h$ converges weakly in $L^1(\Omega; \Mnn)$. 
 This determines the (sub)sequence $(u_h)_h$ for which we are going to prove the result, fixed from now on. 
First notice that the set $A$ defined in \eqref{eq:defA} (in correspondence to the subsequence) is such that $(u_h)_h$ converges pointwise $\mathcal{L}^n$-a.e.\ in $\Omega \sm A$ to a function with finite values (that is in $\Rn$).

We define $\bar{u}\colon \Omega \to (\widetilde{\R})^n$ and $u\colon \Omega\to \Rn$ such that 
\begin{equation}\label{3112171325}
\begin{split}
\bar{u}:=(\tilde{u}^1,\dots, \tilde{u}^n)\,,\quad \text{ where } \tilde{u}^i=\tau^{-1}(\tilde{\tau}_i)\,;\qquad\qquad u:=\bar{u}\,\chi_{\Omega\sm A}\,,
\end{split}
\end{equation}
with the convention that $\tau^{-1} (\pm 1)= \pm \infty$.

The set $A$, which coincides with  $\{x\in \Omega\colon\, \tilde{u}^i(x) \in \{-\infty, +\infty\} \text{ for some }i\in\{1,\dots,n\}\}$,  is measurable, since $\tilde{u}^i(x) \in \R$ if and only if $|\tau(\tilde{u}^i)|<1$ and the functions $\tilde{\tau}_i \colon \Omega\to [-1,1]$ are measurable. 
Since $(u_h)_h$ converges pointwise $\mathcal{L}^n$-a.e.\ in $\Omega \sm A$ to $u$ 
we have that for every $\xi \in \Sn$
\begin{equation}\label{3112171733}
u_h \cdot \xi \rightarrow u\cdot \xi \quad \text{ $\mathcal{L}^n$-a.e.\ in $\Omega \sm A$ }\,.
\end{equation}
Notice that we have not extracted further subsequences depending on $\xi$, and that the limit function $u$ (equal to $\bar{u}$ since we are in $\Omega\sm A$) does not depend on $\xi$.
Eventually, by Lemma~\ref{le:referee2} we have that for $\hn$-a.e.\ $\xi \in \Sn$
\begin{equation}\label{3112171741}
|u_h \cdot \xi| \rightarrow +\infty \quad \mathcal{L}^n\text{-a.e.\ in } A\,.
\end{equation}
\par
\medskip
\paragraph{\bf Lower semicontinuity.}
 Here we prove first \eqref{eq:sciSalto}, which is specific of our approach due to the description of $A$, and then \eqref{eq:convGradSym}, which follows the lines of \cite[Theorem~1.1]{BelCosDM98}.

As in \cite[Theorem~1.1]{BelCosDM98} (see also \cite[Theorem~11.3]{DM13}), we introduce
\begin{equation}\label{0101182132}
\mathrm{I}\xy(u_h):=\int\limits_\Oxy \phi\big(|(\dot{u}_h)\xy|\big)\,\mathrm{d}t\,,
\end{equation}
where 
$(\dot{u}_h)\xy$ is the density of the absolutely continuous part of $\mathrm{D}(\widehat{u}_h)\xy$, the distributional derivative of $(\widehat{u}_h)\xy$ ($(\widehat{u}_h)\xy\in SBV_{\mathrm{loc}}(\Oxy)$ for every $\xi\in \Sn$ and for $\hn$-a.e.\ $y\in \Pi^\xi$, since $u_h\in GSBD(\Omega)$).
Thus for any $\xi \in \Sn$ it holds that
\begin{equation}\label{0101182137}
\int \limits_{\Pi^\xi} \mathrm{I}\xy(u_h) \,\mathrm{d}\hn(y)=\int \limits_\Omega \phi\big(|e(u_h)(x)\xi \cdot \xi  |\big) \leq \int \limits_\Omega \phi\big(|e(u_h)| \big)\dx \leq M\,,
\end{equation}
by Fubini-Tonelli's theorem and \eqref{eq:boundGSBD}, recalling that $\phi$ is non-decreasing.
Moreover, since $u_h \in GSBD(\Omega)$, $\mathrm{D}_\xi\big(\tau(u_h\cdot \xi)\big) \in \mathcal{M}_b^+(\Omega)$
for every $\xi \in \Sn$ and
\begin{equation}\label{0201181304}
\int \limits_{\Pi^\xi} |\mathrm{D}\big(\tau(u_h\cdot \xi)\xy\big)| (\Oxy) \,\mathrm{d}\hn(y)
=|\mathrm{D}_\xi\big(\tau(u_h\cdot \xi)\big)|(\Omega)\leq  M\,,
\end{equation}
by \eqref{0201181308}
and \eqref{eq:boundGSBD}.
We denote
\begin{equation}\label{0201181313}
\mathrm{II}\xy(u_h):=|\mathrm{D}\big(\tau(u_h\cdot \xi)\xy\big)| (\Oxy)\,.
\end{equation}
Let $(u_k)_k=(u_{h_k})_k$ be a subsequence of $(u_h)_h$ such that
\begin{equation}\label{0101182244}
\lim_{k\to \infty} \hn(J_{u_k})=\liminf_{h\to \infty} \hn(J_{u_h})<+\infty\,,
\end{equation}
so that, 
by \eqref{0101182248}, \eqref{0101182137}, 
and Fatou's lemma, we have that for $\hn$-a.e.\ $\xi \in \Sn$
\begin{equation}\label{0101182302}
\liminf_{k\to\infty} \int\limits_{\Pi^\xi}\Big[ \ho\big(J_{(\widehat{u}_k)\xy}\big) + \varepsilon \big( \mathrm{I}\xy(u_k)  + \mathrm{II}\xy(u_k) \big) \Big]\,\mathrm{d}\hn(y) <+\infty\,,
\end{equation}
for a fixed $\varepsilon\in (0,1)$.
Let us fix $\xi \in \Sn$ such that \eqref{3112171741} and \eqref{0101182302} hold. Then there is a subsequence $(u_m)_m=(u_{k_m})_m$ of $(u_k)_k$, depending on $\varepsilon$ and $\xi$, such that 
\begin{equation}\label{0101182324}
\begin{split}
\lim_{m\to\infty} &\int\limits_{\Pi^\xi}\Big[ \ho\big(J_{(\widehat{u}_m)\xy}\big) + \varepsilon \big(\mathrm{I}\xy(u_m)+ \mathrm{II}\xy(u_m) \big)  \Big]\,\mathrm{d}\hn(y)\\
&=\liminf_{k\to\infty} \int\limits_{\Pi^\xi}\Big[ \ho\big(J_{(\widehat{u}_k)\xy}\big) + \varepsilon \big( \mathrm{I}\xy(u_k) + \mathrm{II}\xy(u_k) \big) \Big]\,\mathrm{d}\hn(y)\,.
\end{split}
\end{equation}
Therefore, by \eqref{0101182324},  \eqref{3112171733}, and \eqref{3112171741}, employing  Fatou's lemma, we have that for $\hn$-a.e.\ $y\in \Pi^\xi$
\begin{equation}\label{0101182328}
\liminf_{m\to \infty} \Big[ \ho\big(J_{(\widehat{u}_m)\xy}\big) + \varepsilon \big(\mathrm{I}\xy(u_m) + \mathrm{II}\xy(u_m) \big) \Big] < +\infty\,,
\end{equation}
\begin{equation}\label{0101182331}
(\widehat{u}_m)\xy \rightarrow \widehat{u}\xy \quad \mathcal{L}^1\text{-a.e.\ in }\Dxy\,\qquad |(\widehat{u}_m)\xy| \rightarrow \infty\,, \quad\mathcal{L}^1\text{-a.e.\ in }\Axy\,,
\end{equation}
and
\begin{equation}\label{0201181219}
\tau(u_m\cdot \xi)\xy\rightarrow \tilde{\tau}\xy \quad\text{in }L^1(\Oxy)\,,
\end{equation}
for a suitable $\tilde{\tau}\xy \in L^1(\Oxy)$. Now we employ \eqref{3112171733}, \eqref{3112171741}, and \eqref{0101182331}, \eqref{0201181219} to get
\begin{equation}\label{0201181338}
\begin{cases}
\hspace{-1em}&\tilde{\tau}\xy=\tau(u\cdot \xi)\xy \quad\mathcal{L}^1\text{-a.e.\ in } \Dxy\\
\hspace{-1em}&  |\tilde{\tau}\xy|=1 \quad \mathcal{L}^1\text{-a.e.\ in } \Axy\,.
\end{cases}
\end{equation}

Fixed $y\in \Pi^\xi$ satisfying \eqref{0101182328} and \eqref{0101182331}, and such that 
$(\widehat{u}_m)\xy\in SBV_{\mathrm{loc}}(\Oxy)$
 for every $m$, we extract a subsequence $(u_j)_j=(u_{m_j})_j$ from $(u_m)_m$, depending also on $y$, for which
\begin{equation}\label{0101182346}
\lim_{j\to \infty} \Big[ \ho\big(J_{(\widehat{u}_j)\xy}\big) + \varepsilon \big(\mathrm{I}\xy(u_j)+\mathrm{II}\xy(u_j) \big) \Big]=\liminf_{m\to \infty} \Big[ \ho\big(J_{(\widehat{u}_m)\xy}\big) + \varepsilon \big( \mathrm{I}\xy(u_m)+\mathrm{II}\xy(u_m) \big)\Big]\,.
\end{equation}
Then by \eqref{0201181219} we have that
\begin{equation}\label{0201181346}
\tau(u_j\cdot \xi)\xy \wstar \tilde{\tau}\xy \quad\text{in }SBV(\Oxy)\,.
\end{equation}


In order to describe the set $A$, we consider its slices $\Axy$ and prove that for $\hn$-a.e.\ $y\in \Pi^\xi$
\begin{equation}\label{1307181827}
\Axy \text{ is a finite union of intervals where $\tilde{\tau}\xy$ has either the value $1$ or $-1$}\,,
\end{equation}
and

\begin{equation}\label{0301181024}
\partial A\xy \subset J_{\tilde{\tau}\xy}\,.
\end{equation}
 Recalling that $|\tilde{\tau}\xy|<1$ in $\Dxy$, by \eqref{0201181338}, the property above states that there is a jump each time one passes from values of $\tilde{\tau}\xy$ with absolute value less than 1 to $\Axy$, that is the set where $|\tilde{\tau}\xy|=1$. In terms of the slices of $u$, one passes from finite to infinite values.

Let us show the claimed properties. Up to considering a subsequence of $(\widehat{u}_j)\xy$, we may assume that for every $j$ 
\[
\ho\big(J_{(\widehat{u}_j)\xy}\big)=N_y \in \N\,,
\]
namely there is a fixed
number $N_y$ of jump points. These points tend to $M_y \leq N_y$ points 
\[t_1,\dots ,t_{M_y}\,.\]
Then (recall that $\mathrm{II}\xy(u_j)$ is equibounded in $j$ by \eqref{0101182346}) 
for 
every $l=1,\dots,M_y-1$
\begin{equation*}
\tau(u_j\cdot \xi)\xy  \weak \tilde{\tau}\xy \quad\text{ in } W^{1,1}_{\mathrm{loc}}(t_l, t_{l+1})\,,
\end{equation*}
and the convergence above is locally uniform (for the precise representatives).  Moreover, since $\mathrm{I}\xy(u_j)$ is equibounded again by \eqref{0101182346},
it follows that $x \mapsto (\widehat{u}_j)\xy(x) - (\widehat{u}_j)\xy(\ol x)$ is locally uniformly bounded in $(t_l, t_{l+1})$, for any choice of $\ol x \in (t_l, t_{l+1})$ (by the Fundamental Theorem of Calculus).
Hence for any $l$ 
we have two alternative possibilities: 
\begin{itemize}
\item there is $\ol x \in (t_l, t_{l+1})$ such that \[\lim_{j\to \infty} (\widehat{u}_j)\xy(\ol x)= \widehat{u}\xy(\ol x) \in \R\] (that is $\ol x \notin \Axy$), and then $(\widehat{u}_j)\xy$ converge locally uniformly in $(t_l, t_{l+1})$ to $\widehat{u}\xy$;
\item for $\mathcal{L}^1$-a.e.\  $x \in (t_l, t_{l+1})$,
\[ \lim_{j\to \infty} |(\widehat{u}_j)\xy(x)| = \infty \,,\]
that is $(t_l, t_{l+1}) \subset \Axy$.
\end{itemize}  
Therefore any $(t_l, t_{l+1})$ is contained either in $\Dxy$ or in $\Axy$. Moreover, in the first case we have that $\widehat{u}\xy \in W^{1,1}(t_l, t_{l+1}) \subset L^\infty(t_l, t_{l+1})$. In particular, in this case there is $\eta \in (0,1)$ such that 
\begin{equation}\label{1307182121}
\tilde{\tau}\xy(t_l, t_{l+1})\subset [-1+\eta, 1-\eta]\,.
\end{equation}
This implies \eqref{1307181827} and \eqref{0301181024}.

By \eqref{0101182346}, \eqref{0201181346},  \eqref{0301181024}, and since the jump sets of $\tau(u_j\cdot \xi)\xy$ and $(\widehat{u}_j)\xy$ coincide,
we deduce,  by lower semicontinuity for $SBV$ functions defined in one-dimensional domains
(see \cite[Proposition~4.2]{Amb89UMI}),  that
\begin{equation}\label{0201181710}
\begin{split}
\ho\big(J_{\widehat{u}\xy} \cap \Dxy\big) + \ho\big(\partial A\xy\big) &\leq \ho(J_{\tilde{\tau}\xy})
\\&\leq \liminf_{m\to \infty} \Big[ \ho\big(J_{(\widehat{u}_m)\xy}\big) + \varepsilon \big( \mathrm{I}\xy(u_m)+\mathrm{II}\xy(u_m) \big)\Big]\,.
\end{split}
\end{equation}
We now integrate over $y\in\Pi^\xi$ and use Fatou's lemma with \eqref{0101182324} to get
\begin{equation}\label{0201181712}
\begin{split}
\int \limits_{\Pi^\xi} &\Big[\ho\big(J_{\widehat{u}\xy} \cap \Dxy\big) + \ho\big(\partial A\xy\big)\Big] \,\mathrm{d}\hn(y) \\&\leq \liminf_{k\to\infty} \int\limits_{\Pi^\xi}\Big[ \ho\big(J_{(\widehat{u}_k)\xy}\big) + \varepsilon \big( \mathrm{I}\xy(u_k) + \mathrm{II}\xy(u_k) \big) \Big]\,\mathrm{d}\hn(y)
\end{split}
\end{equation}
for $\hn$-a.e.\ $\xi \in \Sn$.
 In particular we deduce that $A$ has finite perimeter (\textit{cf.}~\cite[Remark~3.104]{AFP}). 
%
%

%
%

We integrate \eqref{0201181712} over $\xi \in \Sn$; by \eqref{0101182248}, \eqref{0101182137}, \eqref{0201181304}, and \eqref{0101182244} we get
\begin{equation}\label{0201181718}
\hn(J_u \cap (\Omega\setminus A)) + \hn(\partial^*A)\leq C\,M \varepsilon + \liminf_{h\to \infty} \hn(J_{u_h})\,,
\end{equation}
for a universal constant $C$. By the arbitrariness of $\varepsilon$ and the definition of $u$ we obtain \eqref{eq:sciSalto}. 
\medskip
\\
 The property \eqref{eq:convGradSym} follows by an adaptation of the arguments in \cite[Theorem~1.1]{BelCosDM98} as in \cite[Theorem~11.3]{DM13}
(which employ Ambrosio-Dal Maso's~\cite[Prop.~4.4]{Amb89UMI}). We report the proof for the reader's convenience.

Fatou's lemma and \eqref{0101182248} give that for $\hn$-a.e.\ $\xi \in \Sn$
\begin{equation}\label{1407180820}
\liminf_{h\to \infty} \int \limits_{\Pi^\xi} \ho(J_{(\widehat{u}_h)\xy}\cap \Omega\xy)\,\mathrm{d}\hn(y) < + \infty\,.
\end{equation}
In particular there is a basis $\{\xi_1, \dots, \xi_n\}$ of $\Rn$ such that this holds for every $\xi$ of the form $\xi=\xi_i+\xi_j$, $i,j=1,\dots, n$. We fix a $\xi$ of this type, and we find a subsequence $(u_k)_k= (u_{h_k})_k$ of $(u_h)_h$, depending on $\xi$, such that 
\begin{equation}\label{1307182046}
\lim_{k\to \infty} \int \limits_{\Pi^\xi} \ho(J_{(\widehat{u}_k)\xy}\cap \Omega\xy)\,\mathrm{d}\hn(y) =\liminf_{h\to \infty} \int \limits_{\Pi^\xi} \ho(J_{(\widehat{u}_h)\xy}\cap \Omega\xy)\,\mathrm{d}\hn(y) \,.
\end{equation}
For a given $w \in L^1(\Omega)$ let (recall \eqref{0101182132} for the definition of $(\dot{u}_k)\xy$)
\begin{equation*}
\mathrm{III}\xy(u_k, w):= \int \limits_{\Dxy} \big| (\dot{u}_k)\xy - w\xy \big| \, \mathrm{d}t\,.
\end{equation*}
By \eqref{3105171927}, \eqref{eq:boundGSBD} (the sequence $(u_h)_h$ has been fixed before \eqref{3112171325}), and Fubini-Tonelli's theorem there is a subsequence $(u_l)_l = (u_{k_l})_l$ of $(u_k)_k$ such that
\begin{equation}\label{1307182024}
\lim_{l\to \infty} \int \limits_{\Pi^\xi} \mathrm{III}\xy(u_l, w) \dh(y) = \liminf_{k\to \infty} \int \limits_{\Omega \sm A} \big| e(u_k) \xi \cdot \xi - w \big| \dx < + \infty\,.
\end{equation}
Let us also fix $\varepsilon\in (0,1)$. Again by Fubini-Tonelli's theorem, there is a subsequence $(u_m)_m=(u_{l_m})_m$ of $(u_l)_l$, depending on $\xi$, $w$, $\varepsilon$, such that \eqref{0101182331} holds for $\hn$-a.e.\ $y \in \Pi^\xi$ and
\begin{equation}\label{1307182038}
\begin{split}
\lim_{m\to\infty} &\int\limits_{\Pi^\xi} \mathrm{III}\xy (u_m, w) +\varepsilon\big[ \ho\big(J_{(\widehat{u}_m)\xy}\big) + \mathrm{I}\xy(u_m)+ \mathrm{II}\xy(u_m)   \big]\,\mathrm{d}\hn(y)\\
&=\liminf_{l\to\infty} \int\limits_{\Pi^\xi} \mathrm{III}\xy (u_l, w) + \varepsilon \big[ \ho\big(J_{(\widehat{u}_l)\xy}\big) + \mathrm{I}\xy(u_l) + \mathrm{II}\xy(u_l) \big]\,\mathrm{d}\hn(y)\,.
\end{split}
\end{equation}
By \eqref{0101182137}, \eqref{0101182244}, \eqref{1307182046},  \eqref{1307182024}, and Fatou's lemma, for $\hn$-a.e.\ $y\in \Pi^\xi$
\begin{equation}\label{1307182051}
\liminf_{m\to \infty} \Big[\mathrm{III}\xy (u_m, w) +\varepsilon\big[ \ho\big(J_{(\widehat{u}_m)\xy}\big) + \mathrm{I}\xy(u_m)+ \mathrm{II}\xy(u_m) \big] \Big] < +\infty\,.
\end{equation}
Let $y \in \Pi^\xi$ be such that \eqref{0101182331} and \eqref{1307182051} hold, and $(\widehat{u}_m)\xy\in SBV_{\mathrm{loc}}(\Oxy)$ for every $m$. We find a subsequence $(u_j)_j=(u_{m_j})_j$ of $(u_m)_m$, depending also on $y$, for which
\begin{equation}\label{1307182100}
\begin{split}
& \lim_{j\to \infty} \Big[\mathrm{III}\xy (u_j, w) +\varepsilon\big[ \ho\big(J_{(\widehat{u}_j)\xy}\big) + \mathrm{I}\xy(u_j)+ \mathrm{II}\xy(u_j) \big] \Big] \\& \hspace{1em}= \liminf_{m\to \infty} \Big[\mathrm{III}\xy (u_m, w) +\varepsilon\big[ \ho\big(J_{(\widehat{u}_m)\xy}\big) + \mathrm{I}\xy(u_m)+ \mathrm{II}\xy(u_m) \big] \Big] \,.
\end{split}
\end{equation}
Recalling the form of $\Axy$ (and \eqref{1307182121}) we deduce that $(\widehat{u}_j)\xy$ converge to $\widehat{u}\xy$ weakly$^*$ in $BV(I)$ for any $I$ compactly contained in $\Dxy$, and then $(\dot{u}_j)\xy \weak \dot{u}\xy$ in $L^1\big(\Dxy\big)$, by \eqref{eq:equiint}. Together with \eqref{1307182100} this gives
\begin{equation*}
\begin{split}
\mathrm{III}\xy(u, w)  \leq \liminf_{j\to \infty} \mathrm{III}\xy(u_j, w) \leq  \liminf_{m\to \infty} \Big[\mathrm{III}\xy (u_m, w) +\varepsilon\big[ \ho\big(J_{(\widehat{u}_m)\xy}\big) + \mathrm{I}\xy(u_m)+ \mathrm{II}\xy(u_m) \big] \Big]\,. 
\end{split}
\end{equation*}
Integrating with respect to $y\in \Pi^\xi$, by Fatou's lemma and \eqref{1307182024}, \eqref{1307182038} plus the bounds \eqref{0101182137}, \eqref{0201181304}, \eqref{0101182302}, we get 
\begin{equation*}
\int \limits_{\Omega\sm A} \big| e(u) \xi \cdot \xi - w  \big| \leq \liminf_{k\to \infty} \int \limits_{\Omega \sm A} \big| e(u_k) \xi \cdot \xi - w \big| \dx + \varepsilon\Big( C\, M + \liminf_{h\to \infty} \int \limits_{\Pi^\xi} \ho(J_{(\widehat{u}_h)\xy}\cap \Omega\xy)\,\mathrm{d}\hn(y)\Big)\,.
\end{equation*}
By \eqref{1407180820} and the arbitrariness of $\varepsilon$, we deduce that for all $w\in L^1(\Omega)$,
\begin{equation*}
\int \limits_{\Omega\sm A} \big| e(u) \xi \cdot \xi - w  \big| \leq \liminf_{k\to \infty} \int \limits_{\Omega \sm A} \big| e(u_k) \xi \cdot \xi - w \big| \dx \,.
\end{equation*}
Since the sequence $\big(e(u_h)\big)_h$ weakly converges in $L^1(\Omega \sm A; \Mnn)$, then 
 \cite[Proposition~4.4]{Amb89UMI} gives
\begin{equation*}\label{1307182156}
e(u_h) \xi \cdot \xi \weak e(u) \xi \cdot \xi \quad \text{ in } L^1(\Omega \sm A)\,,
\end{equation*}
and by the arbitrariness of 
$\xi=\xi_i + \xi_j$
we deduce \eqref{eq:convGradSym}. 


%
\par
\medskip
\paragraph{\bf Closure.}  We now show that the limit function $u$, defined in \eqref{3112171325}, is in $GSBD(\Omega)$.

 Employing \eqref{0201181308} and recalling \eqref{eq:boundGSBD}, we have that there exist $\lambda_{u_h}\in \mathcal{M}_b^+(\Omega)$ such that 
\[
\lambda_{u_h}(\Omega) \leq  M \,,
\]
and for every $\xi \in \Sn$ and every Borel set $B\subset \Omega$
\[
|\mathrm{D}_\xi\big(\tau(u_h\cdot \xi)\big)|(B)\leq \lambda_{u_h}(B)\,.
\]

Let $\tilde{\lambda} \in \mathcal{M}_b^+(\Omega)$ be a weak$^*$ limit of a subsequence of $(\lambda_{u_h})_h$, so that $\tilde{\lambda}(\Omega) \leq M$.
%
Notice that 
\begin{equation}\label{0401181312}
\mathrm{D}_\xi\tau(u \cdot \xi) \in \mathcal{M}_b(\Omega) \quad\text{for every }\xi \in \Sn
\end{equation}
and 
\begin{equation}\label{0401181316}
|\mathrm{D}_\xi\tau(\tilde{u} \cdot \xi)|(B) \leq \tilde{\lambda}(B)=: \lambda_u(B)
\end{equation}
for every Borel set $B\subset \Omega$, where $\tilde{\lambda}$ has been defined above. This follows by a slicing procedure and the use of Fatou's lemma for every $\xi$, to reconstruct at the end $|\mathrm{D}_\xi(\tau(u\cdot \xi))|(\Omega)$ from $\mathrm{II}\xy(u):=|\mathrm{D}\big(\tau(u\cdot \xi)\xy\big)| (\Oxy)$ (see \eqref{0201181313}), as in \eqref{0201181304}.
The important point here is to get the semicontinuity 
\begin{equation*}
\mathrm{II}\xy(u)\leq \liminf_{j\to \infty} \mathrm{II}\xy(u_j)= \liminf_{j\to \infty}  |\mathrm{D}\big(\tau(u_j\cdot \xi)\xy \big)|(\Oxy)\,,
\end{equation*}
for the slices, which follows from \eqref{0201181346}. Indeed $\mathrm{II}\xy(u) \leq |\mathrm{D}\big(\tilde{\tau}\xy) \big)|(\Oxy)$ because  
 $\tau(u\cdot \xi)\xy=\tilde{\tau}\xy$ in $(\Omega \setminus A)\xy$ by \eqref{0201181338} and 
$\tau(u\cdot \xi)=0$ in $A\xy$, so we employ \eqref{0301181024}. Moreover, it is immediate that $\widehat{u}\xy \in SBV_{\mathrm{loc}}(\Omega\xy)$.
Therefore $u\in GSBD(\Omega)$.
This concludes the proof.
\end{proof}

\section{Existence for minimisers of the Griffith energy}
\label{sec:Gri}
\addtocontents{toc}{\protect\setcounter{tocdepth}{1}}
Employing Theorem~\ref{thm:main}, we deduce in this section the existence of weak solutions to the minimisation problem of the Griffith energy with Dirichlet boundary conditions.
\subsection*{Existence of weak solutions}
Assume $\Omega\subset \Rn$ be an open, bounded
 domain for which \[\dom=\dod\cup \don\cup N\,,\] with $\dod$ and $\don$ relatively open, $\dod \cap \don =\emptyset$, $\mathcal{H}^{n-1}(N)=0$, 
$\dod \neq \emptyset$, and $\partial(\dod)=\partial(\don)$. 
 Let $u_0\in W^{1,p}(\Rn;\Rn)$ and $W\colon \Mnn\to [0,\infty)$ be convex, with $W(0)=0$ and
\begin{equation}\label{eq:equiintW}
W(\xi) \geq \phi(|\xi|) \qquad \text{for }\xi \in \Mnn\,,
\end{equation}
where $\phi$ satisfies \eqref{eq:equiint}.

Let $K_0 \subset \Omega \cup \dod$ be $(n{-}1)$-countably rectifiable with $\hn(K_0)<+\infty$, and consider the minimisation problem:
\begin{equation}\label{2301181120}
\min_{v\in GSBD(\Omega)} \Bigg\{ \int\limits_\Omega W(e(v)) \dx + \hn\big(J_v \cup(\dod \cap \{\mathrm{tr}_\Omega \, v \neq \mathrm{tr}_\Omega\, u_0\}) \setminus K_0\big) 
\Bigg\}\,.
\end{equation}
Notice that, defining 
$\widetilde{\Omega}:=\Omega \cup U$, where $U$ is an open bounded set with $U \cap \dom =\dod$, we can recast the problem as
\begin{equation}\label{2301181148}
\min_{v\in GSBD(\widetilde{\Omega})} \Bigg\{ \int\limits_{\widetilde{\Omega}} W(e(v)) \dx + \hn(J_v \setminus K_0) \colon v=u_0 \text{ in }\widetilde{\Omega}\setminus (\Omega\cup \dod)
\Bigg\}\,.
\end{equation}
Then we have the following existence result.
\begin{theorem}\label{teo:existminGrif}
Problem \eqref{2301181148} admits solutions.
\end{theorem}
\begin{proof}
Let $u_h\in GSBD(\widetilde{\Omega})$ with $u_h=u_0 \text{ in }\widetilde{\Omega}\setminus (\Omega\cup \dod)$ be the elements of a minimising sequence for \eqref{2301181148}. Observe that the infimum of problem \eqref{2301181148} is finite, since the functional is nonnegative and $u_0$ is an admissible competitor.

Assume for the moment that $K_0$ is compact.
By \eqref{eq:equiintW} the functions $u_h$ satisfy the hypotheses of Theorem~\ref{thm:main} with $\Omega=\widetilde{\Omega}\setminus K_0$, so that
 there exist $A \subset \widetilde{\Omega}\setminus K_0$ with finite perimeter and a measurable function $u \colon \widetilde{\Omega}\setminus K_0 \to \Rn$ with $u=0$ in $A$
 such that (up to a subsequence)
 \begin{equation}\label{0702181209}
 A=\{x\in \widetilde{\Omega}\setminus K_0 \colon |u_h(x)|\to\infty\},\qquad u_h\to u \quad \mathcal{L}^n\text{-a.e.\ in }\Omega\sm (K_0 \cup A)
 \end{equation}
  (since $\mathcal{L}^n(K_0)=0$ we could consider just $\widetilde{\Omega}$ above, but we keep $\widetilde{\Omega}\sm K_0$ to indicate the set where we apply Theorem~\ref{thm:main}) and 
 \begin{equation*}
 \int\limits_{\widetilde{\Omega}} W(e(u)) \dx + \hn(J_{u} \setminus K_0)\leq \liminf_{h\to \infty} \int\limits_{\widetilde{\Omega}} W(e(u_h)) \dx + \hn(J_{u_h} \setminus K_0)\,,
 \end{equation*}
 Moreover, by \eqref{0702181209} and the admissibility condition for $u_h$ it follows that
$u=u_0$ in $\widetilde{\Omega}\setminus (\Omega\cup \dod)$, and in particular $A$ does not intersect $\big(\widetilde{\Omega}\setminus (\Omega\cup \dod)\big)$. Since $W$ is convex, we have lower semicontinuity for the bulk term, and $u$ solves \eqref{2301181148}. This proves the theorem if $K_0$ is compact.  Notice that this holds for any other function $v$ which coincides with $u$ in $\Omega\sm A$ 
and is set equal to any fixed infinitesimal rigid motion in $A$, since the energy of $v$ in $A$ is null, and then by \eqref{1407180906} the Griffith energy of $v$ is less than the the $\liminf$ of the energies of $u_h$.



If $K_0$ is not compact, for any $\varepsilon>0$ consider $\widehat{K_0}\subset K_0$ with $\hn(K_0\sm \widehat{K_0})<\varepsilon$. Then, arguing as above for the open set $\widetilde{\Omega}\sm \widehat{K_0} \supset \widetilde{\Omega}\sm K_0$, we get still 
\[
\int\limits_{\widetilde{\Omega}} W(e(u)) \dx \leq \liminf_{h\to \infty} \int\limits_{\widetilde{\Omega}} W(e(u_h)) \dx \,,
\]
and
\begin{equation*}
\begin{split}
\hn(J_{u}\sm K_0)&\leq \hn(J_{u}\sm \widehat{K_0}) \leq \liminf_{h\to\infty} \hn(J_{u_h}\sm \widehat{K_0}) \\
&\leq \liminf_{h\to \infty} \hn(J_{u_h}\sm K_0) + \hn(K_0\sm \widehat{K_0})< \liminf_{h\to \infty} \hn(J_{u_h}\sm K_0) + \varepsilon\,,
\end{split}
\end{equation*}
since $J_{u}\sm K_0 \subset J_{u}\sm \widehat{K_0}$ and $J_{u_h}\sm \widehat{K_0} \subset (J_{u_h}\sm K_0) \cup (K_0\sm \widehat{K_0})$ (\textit{cf.}~also \cite[Theorem~2.5]{FriSol16}). We conclude since $\varepsilon>0$ is arbitrary.
\end{proof}

\begin{remark}\label{0702182040}
Since, as observed in the proof, a family of minimisers is obtained by adding any fixed infinitesimal rigid motion in $A$ to a given minimiser, 
we conclude that 
\[
\hn(\partial^*A\cap \{\mathrm{tr}\,u=a\})=0
\]
 for every
infinitesimal rigid motion $a$ ($a(x)=\mathbf{a}\cdot x+b$, $\mathbf{a}+\mathbf{a}^T=0$),
 where $\mathrm{tr}$ denotes here the trace of $u$ on $\partial^*A$ (which is $(n{-}1)$-countably rectifiable) from $\Omega\sm A$.
\end{remark}
\subsection*{Existence of strong solutions}
In recent works~\cite{CFI17DCL,ChaConIur17}, Chambolle, Conti, Focardi, and Iurlano have shown more regularity for the 
possible minimisers of
\eqref{2301181148} (or \eqref{2301181120}) 
if $W(\xi)=\mathbb{C}e(\xi)\colon e(\xi)$ (in \cite{ChaConIur17}), or $n=2$ and
\begin{equation}\label{2501182258}
W(\xi)=f_\mu(\xi):=\frac{1}{p}\Big((\mathbb{C}\xi\colon \xi + \mu)^{p/2}-\mu^{p/2}\Big)
\end{equation} 
(in \cite{CFI17DCL}), requiring that $\mathbb{C}\colon \Mnn \to \Mnn$ is a symmetric linear map with
\begin{equation*}
\mathbb{C}(\xi-\xi^T)=0 \quad\text{and}\quad \mathbb{C}\xi\cdot \xi\geq c_0 |\xi+\xi^T|^2 \qquad \text{for all }\xi\in \Mnn\,.
\end{equation*}
More precisely, the essential closedness of the jump set is established:
\begin{theorem}\label{teo:CCIF}
Let $K_0 \subset \Omega \cup \dod$ closed, with $\hn(K_0)< + \infty$, and $u\in GSBD^2(\Omega\setminus K_0)$ (or $u\in GSBD^p(\Omega\sm K_0)$, if $\Omega\subset \R^2$) be a minimiser of 
\begin{equation}\label{1407180937}
\int\limits_\Omega \mathbb{C} e(v)\colon e(v) \dx + \hn\big(J_v \cup(\dod \cap \{\mathrm{tr}_\Omega \, v \neq \mathrm{tr}_\Omega\, u_0\}) \setminus K_0\big)
\end{equation} 
(a minimiser of \eqref{2301181148} with \eqref{2501182258}, respectively). Then 
\begin{equation*}
\hn\big((\Omega \setminus K_0)\cap (\overline{J}_u \sm J_u)\big)=0\,,\qquad u\in C^1\big(\Omega\sm(K_0 \cup \overline{J}_u)\big)\,.
\end{equation*}
\end{theorem}
 
In \cite{CC19}, this is extended 
to $\Omega \cup \dod$, yielding the following result (see~\cite{BabGia14} for the $SBV$ case):
\begin{theorem}
Let $\dod$ be of class $C^1$, $u_0 \in W^{1, \infty}(\Rn; \Rn)$, and $u\in GSBD^2(\Omega\setminus K_0)$, 
be a minimiser of \eqref{1407180937}
Then 
\begin{equation*}
\hn\big(((\Omega \cup \dod) \setminus K_0)\cap (\overline{J}_u \sm J_u)\big)=0\,,\qquad u - u_0 \in C^1\big((\Omega \cup \dod) \sm(K_0 \cup \overline{J}_u)\big)\,.
\end{equation*}
\end{theorem}

Another consequence of Theorem~\ref{thm:main} is a compactness result for phase-field approximations of~\eqref{eq:FM}, which are used for the numerical simulations of evolutions in brittle fracture (such as in~\cite{BouFraMar00}), see~\cite{ChaCri17} for details.

\bigskip
\noindent {\bf Acknowledgements.} 
V.~Crismale is supported by a grant of the LabEx LMH (ANR-11-LABX-0056-LMH, ``\emph{Investissement d'avenir}).
This work was also included in V.~Crismale's project \emph{BriCoFra}, Marie Sk\l odowska-Curie Standard European Fellowship No 793018 starting October 2018.
 The authors wish to thank the anonymous referees for their valuable comments. 

%
%
%


\end{document}